\documentclass[letterpaper, 10 pt, conference]{ieeeconf}  

\IEEEoverridecommandlockouts                              

\usepackage{caption}
\usepackage{subcaption}
\usepackage{graphicx}
\usepackage{algorithm}
\usepackage{algorithmicx}
\usepackage{algpseudocode}
\usepackage{amsmath} 
\usepackage{amssymb}  

\graphicspath{{figures/}}

\newcommand{\etal}{\textit{et al}.~}
\newcommand{\ie}{\textit{i}.\textit{e}.}
\newcommand{\eg}{\textit{e}.\textit{g}.}

\newtheorem{definition}{Definition}
\newtheorem{remark}{Remark}
\newtheorem{lemma}{Lemma}
\newtheorem{theorem}{Theorem}
\newtheorem{example}{Example}

\title{\LARGE \bf
	A Graph-based Conflict-free Cooperation Method \\for Intelligent Electric Vehicles at Unsignalized Intersections
}

\author{Chaoyi Chen$^{1}$, Qing Xu$^{1}$, Mengchi Cai$^{1}$, Jiawei Wang$^{1}$, Biao Xu$^{2}$, \\ Xiangbin Wu$^{3}$, Jianqiang Wang$^{1}$, Keqiang Li*$^{1}$ and Chunyu Qi$^{4}$.
	\thanks{*This work was supported by the National Key Research and Development Program of China under Grant 2019YFB1600804, the National Natural Science Foundation of China under Grant 52072212, and Intel Collaborative Research Institute Intelligent and Automated Connected Vehicles.}
	\thanks{$^{1}$Chaoyi Chen, Qing Xu, Mengchi Cai, Jianqiang Wang and Keqiang Li are with School of Vehicle and Mobility, Tsinghua University, Beijing 100084, P.R.China}
	\thanks{$^{2}$Biao Xu is with College of Mechanical and Vehicle Engineering, Hunan University}
	\thanks{$^{3}$Xiangbin Wu is with Intel Lab China}
	\thanks{$^{4}$Chunyu Qi is with Beijing BITNEI Corp., Ltd}
	\thanks{Corresponding author: Keqiang Li, Email address:likq@tsinghua.edu.cn}
}

\begin{document}
\maketitle
\thispagestyle{empty}
\pagestyle{empty}

\begin{abstract}
Electric, intelligent, and network are the most important future development directions of automobiles. Intelligent electric vehicles have shown great potentials to improve traffic mobility and reduce emissions, especially at unsignalized intersections. Previous research has shown that vehicle passing order is the key factor in traffic mobility improvement. In this paper, we propose a graph-based cooperation method to formalize the conflict-free scheduling problem at unsignalized intersections. Firstly, conflict directed graphs and coexisting undirected graphs are built to describe the conflict relationship of the vehicles. Then, two graph-based methods are introduced to solve the vehicle passing order. One method is an optimized depth-first spanning tree method which aims to find the local optimal passing order for each vehicle. The other method is a maximum matching algorithm that solves the global optimal problem. The computational complexity of both methods is also derived. Numerical simulation results demonstrate the effectiveness of the proposed algorithms.
\end{abstract}

\section{INTRODUCTION}
Intersections are the most common merging points in urban traffic scenarios~\cite{xu2021coordinated}. According to Federal Highway Administration, each year more than 2.8 million intersection-related crashes occur in the U.S., which hold 44 percent of all the crashes~\cite{azimi2014stip}. Through V2X technology, the centralized controller at the intersection coordinates the intelligent electric vehicles to pass the intersection, which guarantees conflict-free vehicle cooperation. On the other hand, electric vehicles have a natural sort of advantage in vehicles' automated control since their power-train is highly electrified~\cite{chan1997overview}. Thus, electric vehicles have great potentials in improving traffic safety and mobility at intersections~\cite{zheng2020smoothing, chen2020mixed}.


Considering the fully-autonomous scenario, the traffic light becomes unnecessary when all the vehicles share their driving information in real time. In perspective of single vehicle control, multiple methods have been proposed to optimize the vehicle speed trajectory,~\eg, fuzzy logic~\cite{Milanes10}, Model Predictive Control (MPC)~\cite{Zheng17} or optimal control~\cite{malikopoulos2018decentralized}.


In terms of a large scale of vehicles, researchers have pointed out that the vehicles' passing order is the key factor that influences the traffic mobility~\cite{li2006cooperative}. As vehicles approaching the intersection, the optimal passing order changes dynamically and constantly. One straightforward solution is First-in-First-out (FIFO) strategy~\cite{malikopoulos2018decentralized}, which means firstly-entered vehicles are scheduled to leave the intersection first. However, the FIFO passing order is not likely to be the optimal solution in most cases. To tackle this problem, various methods have been proposed to find the optimal passing order,~\eg, spanning tree~\cite{li2006cooperative}, Monte Carlo tree search~\cite{xu2019cooperative} and dynamic programming~\cite{yan2009autonomous}. Since optimal passing order is a discrete rather than a continuous problem, geometry topology is also introduced in system modeling~\cite{lin2019graph}. However, the above-mentioned research mainly focus on specific methods to obtain the passing order and only a few of them analyze the optimality and computational complexity. Some research have focus on the reduction of the optimal passing order problem~\cite{ahn2019abstraction, miculescu2019polling}, but they mainly focus on the feasibility of a conflict-free solution. In practice, however, the traffic efficiency is another critical topic that has not been fully considered in intersection coordination.


Generally speaking, the majority of existing works focus on obtaining the passing order through a specific method without considering the optimality and computational complexity. Some research focuses on the feasibility of a conflict-free solution, but traffic efficiency is neglected. Therefore, a concise and rigid solution of optimal passing order which considers both traffic safety and mobility has yet not been developed. In this paper, we propose a graph-based conflict-free cooperation method. Based on that, two novel methods are introduced to obtain the local or global optimal vehicle passing order.

Specifically, our contributions are:

\begin{enumerate}
	\item An optimized depth-first spanning tree method is proposed based on~\cite{xu2018distributed}. After separating different conflict types by Conflict Directed Graph (CDG), the overall depth of the spanning tree,~\ie, evacuation time of the vehicles, is further reduced. The method has low computational complexity and the passing order solution it generates is proved to be local optimal.
	\item From graphic analysis, Coexisting Undirected Graph (CUG) is introduced to describe vehicles which could pass the intersection simultaneously. A maximum matching method is further applied to obtain the spanning tree, which is proved to be the global optimal passing order solution.
\end{enumerate}

The rest of this paper is organized as follows. Section~\ref{Sec:Scenario} introduces the scenario setup and conflict analysis. Section~\ref{Sec:METHODOLOGY} presents the optimized depth-first spanning tree method and the maximum matching method. Simulation results are in Section~\ref{Sec:Simulation} and Section~\ref{Sec:Conclusion} concludes the paper.

\section{SCENARIO SETUP}
\label{Sec:Scenario}

\label{sec:ConflictAnalysis}
In this paper, we consider a three-lane intersection as shown in Fig.~\ref{fig:Intersection}, where the incoming vehicles are separated into different lanes referring to their destination. Similar to existing research\cite{xu2019cooperative, malikopoulos2018decentralized, xu2018distributed}, lane change behavior is not permitted to guarantee vehicle safety and improve traffic efficiency inside the control zone $ R_c $, as shown as the blue dashed line in Fig.~\ref{fig:Intersection}.

\begin{figure}[tb]
	\centering
	\includegraphics[width=\linewidth]{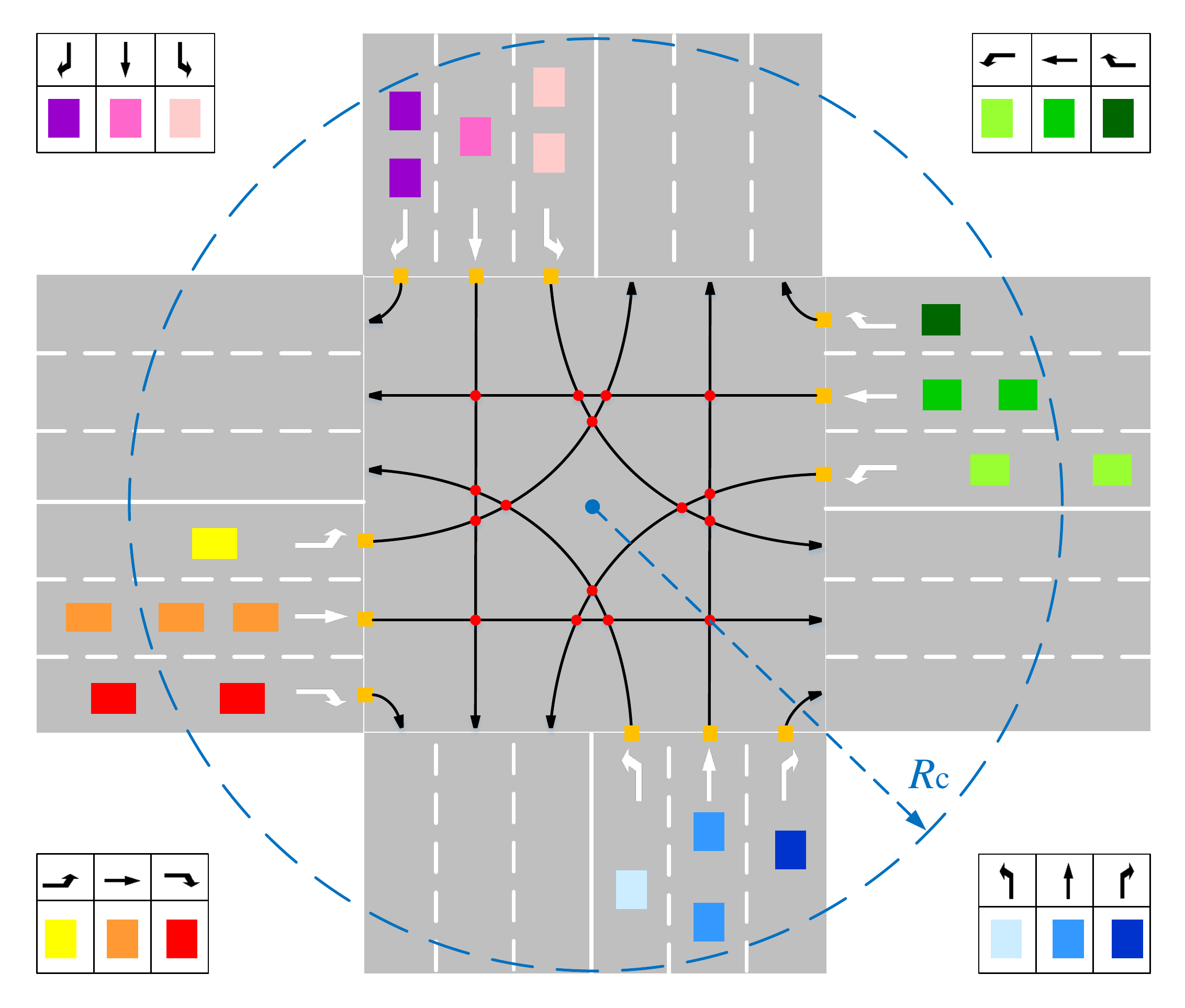}
	\caption{Illustration of the traffic scenario.}
	\label{fig:Intersection}
\end{figure}

A typical intersection has four conflict modes, including crossing, converging, diverging, and no conflict~\cite{xu2018distributed}. 16 red circles in Fig.~\ref{fig:Intersection} represent crossing conflict points, where the vehicles have potentials to collide as shown in Fig.~\ref{fig:ConflictCrossing}. Then, 12 orange squares represent diverging conflict points, where two vehicles come from the same lane as shown in Fig.~\ref{fig:ConflictDiverging}. Converging conflict means the vehicles are going into the same lane, which doesn't exist in our scenario.

\begin{figure}[tb]
	\centering
	\subcaptionbox{Crossing Conflict\label{fig:ConflictCrossing}}
	{\includegraphics[width=0.4\linewidth]{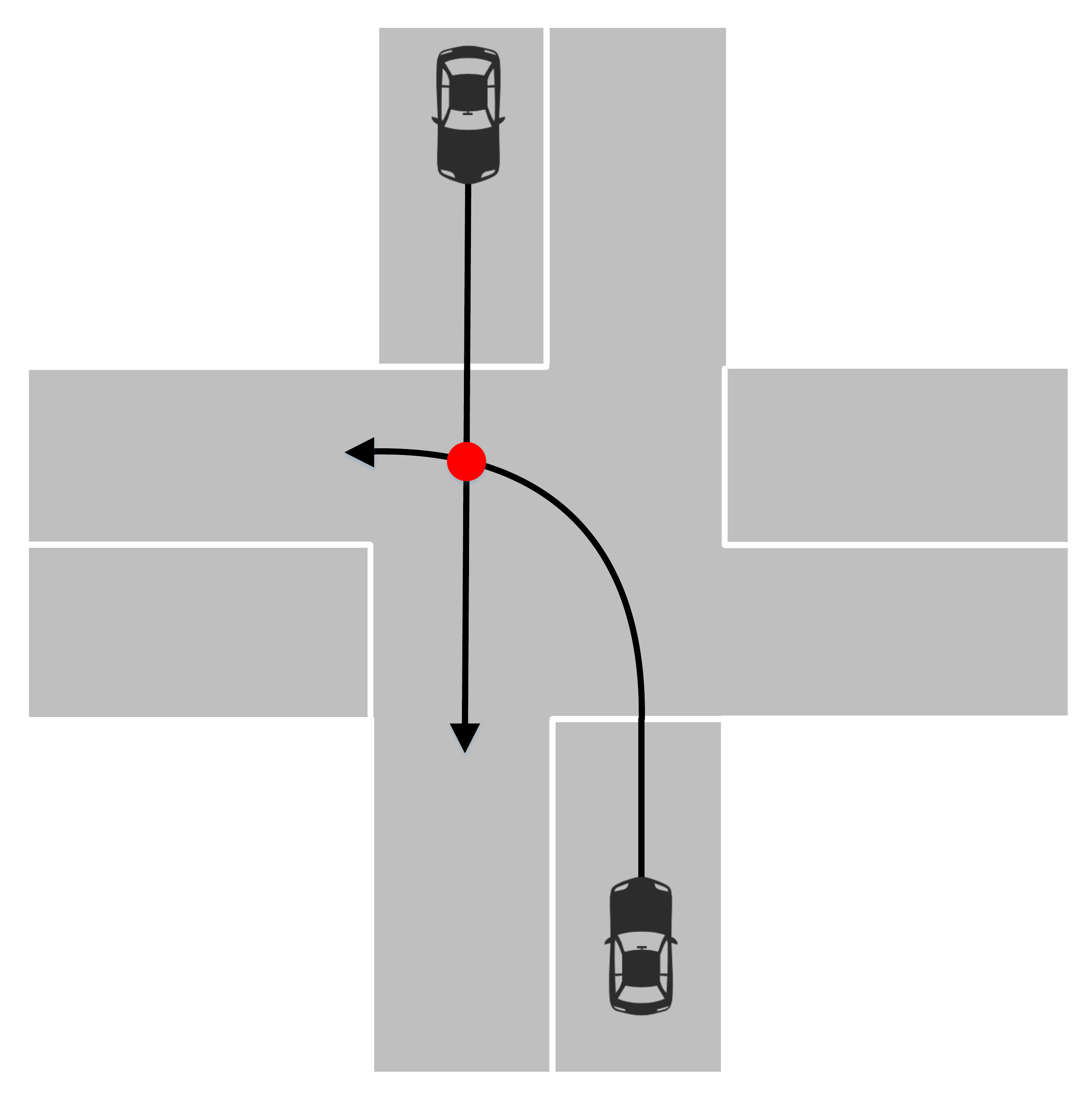}}
	\subcaptionbox{Diverging Conflict\label{fig:ConflictDiverging}}
	{\includegraphics[width=0.4\linewidth]{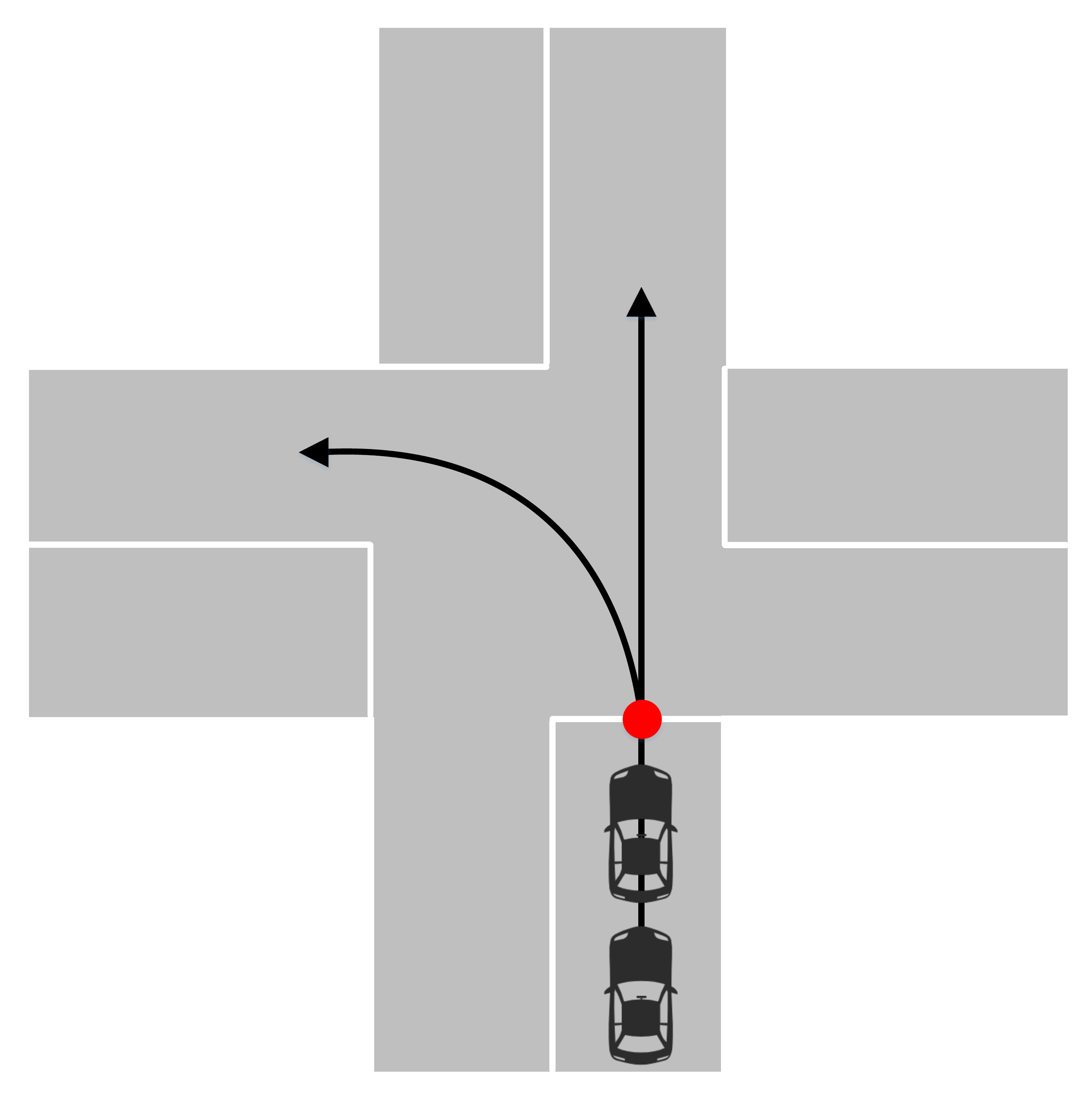}}
	\caption{Different conflict modes.}
	\label{fig:ConflictRelationship}
\end{figure}

The incoming vehicles are indexed from $ 1 $ to $ N $ according to their arriving sequence at the control zone $ R_c $. We define different conflict sets to describe the conflict relationship of the vehicles. For each vehicle $ i \left(i \leq N, i \in \mathbb{N}^+\right) $, its crossing set is defined as $ \mathcal{C}_{i} $, and the diverging set as $ \mathcal{D}_{i} $. Note that because the conflict sets are determined when the vehicle reaches the control zone, vehicle indexes in the conflict sets are smaller than that of the ones on the control zone border. 

\begin{example}
	\label{exp:1}
	Considering an example case as shown in Fig.~\ref{fig:ScenarioSimplify}. Vehicle $ 1 $ and $ 2 $ are coming from the east direction on the middle lane and left lane respectively. Vehicle $ 3 $ and $ 4 $ are approaching on the middle lane from south and west. Vehicle $ 5 $ and $ 6 $ are both on the middle lane from the north.  
\end{example}

In this example, $ \mathcal{C}_{1} = \varnothing $, $ \mathcal{D}_{1} = \varnothing $, $ \mathcal{C}_{2} = \varnothing $, $ \mathcal{D}_{2} = \varnothing $, $ \mathcal{C}_{3} = \{1,2\} $, $ \mathcal{D}_{3} = \varnothing $, $ \mathcal{C}_{4} = \{2,3\} $, $ \mathcal{D}_{4} = \varnothing $, $ \mathcal{C}_{5} = \{1,4\} $, $ \mathcal{D}_{5} = \varnothing $, $ \mathcal{C}_{6} = \{1,4\} $, $ \mathcal{D}_{6} = \{5\} $.

\begin{figure}[tb]
	\centering
	\subcaptionbox{Scenario Case\label{fig:ScenarioSimplify}}
	{\includegraphics[width=0.8\linewidth]{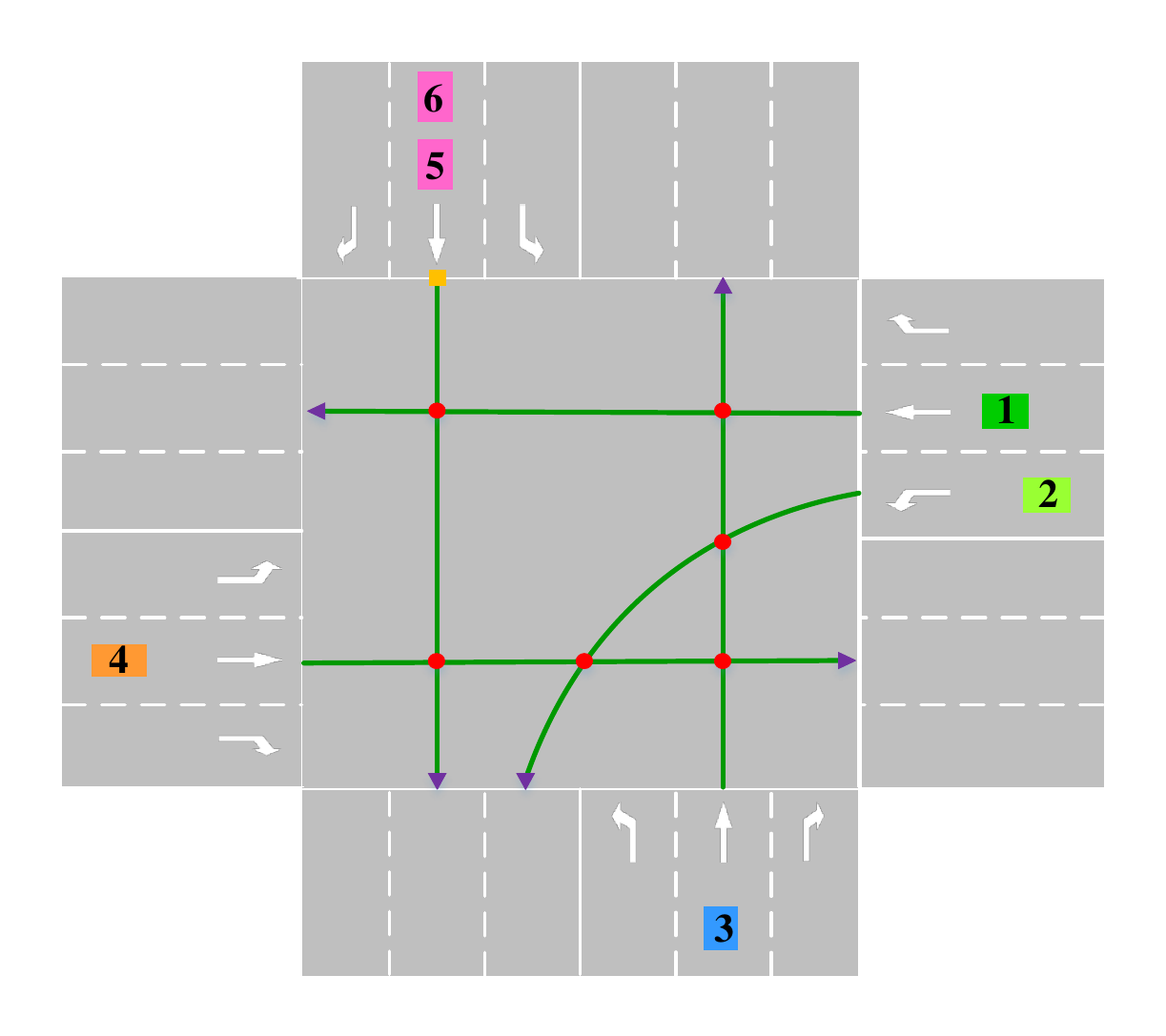}}
	
	\subcaptionbox{Conflict Directed Graph\label{fig:ConflictDirectedGraph}}
	{\includegraphics[width=0.47\linewidth]{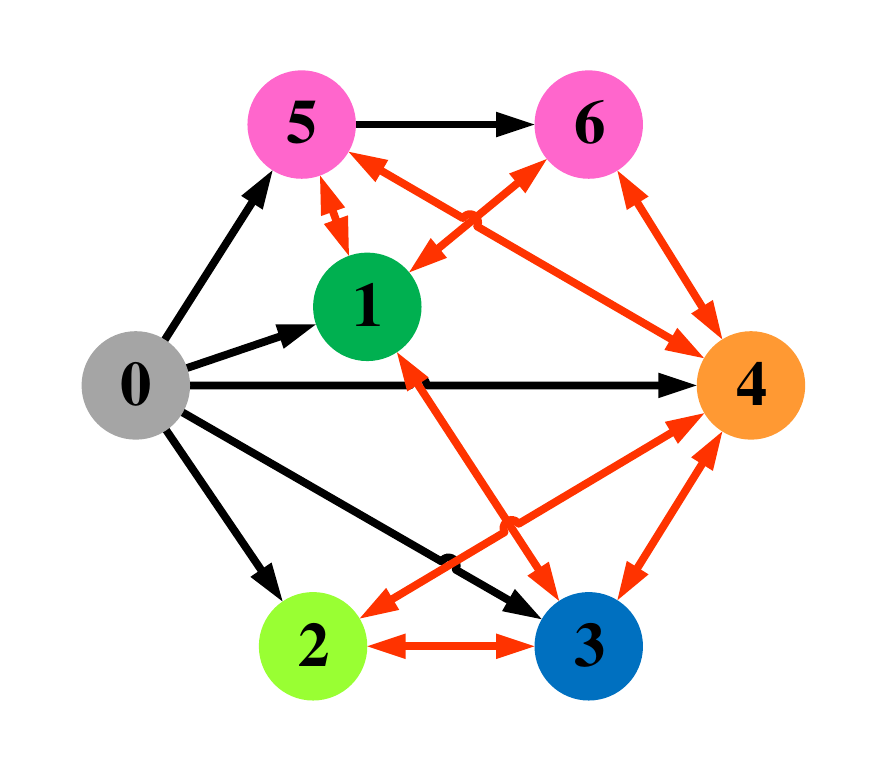}}
	\subcaptionbox{Coexisting Undirected Graph\label{fig:CoexistUndirectedGraph}}
	{\includegraphics[width=0.48\linewidth]{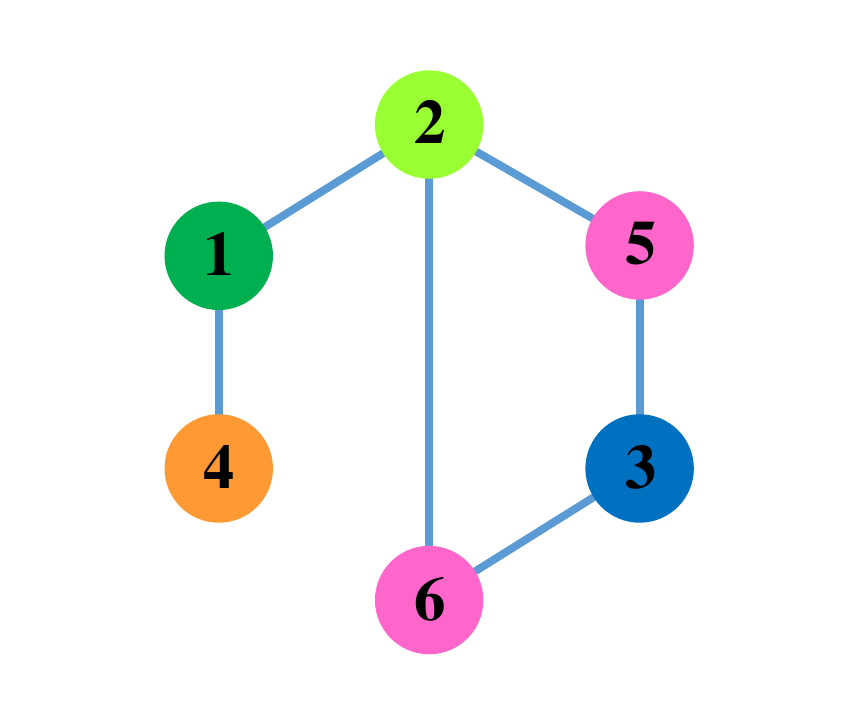}}
	
	\caption{A typical intersection scenario and conflict analysis.}
	\label{fig:ConflictDirectedScenario}
\end{figure}

\section{METHODOLOGY}
\label{Sec:METHODOLOGY}
In~\cite{xu2018distributed}, Xu~\etal propose a method of projecting vehicles from different lanes into a virtual lane, which forms a virtual platoon. In this way, the vehicles from different lanes can drive through the intersection as if they were in the same lane to improve driving smoothness. A Depth-first Spanning Tree (DFST) method is then introduced to find the vehicle passing order. In Section~\ref{Sec:OPT-DFST}, we optimize the DFST method to obtain a local optimal passing order solution, which we name as Optimized Depth-first Spanning Tree (OPT-DFST) method. Then in Section~\ref{Sec:MM}, we introduce the maximum matching (MM) method to find the global optimal passing order solution.

\subsection{Optimized Depth-first Spanning Tree Method}
\label{Sec:OPT-DFST}
In the virtual platoon, there exists a virtual leading vehicle $ 0 $ with the constant moving velocity in front of vehicles that are closest to the intersection. In this way, the leading vehicles in the real world maintain a typical platooning behavior. Accordingly, if the vehicle is the closest to the intersection on each lane, vehicle $ 0 $ is added into its diverging set,~\eg, in Example~\ref{exp:1}, $ \mathcal{D}_{1} = \{0\} $, $ \mathcal{D}_{2} = \{0\} $, $ \mathcal{D}_{3} = \{0\} $, $ \mathcal{D}_{4} = \{0\} $, $ \mathcal{D}_{5} = \{0\} $.

Based on the conflict sets analysis in Section~\ref{sec:ConflictAnalysis}, we further define a Conflict Directed Graph (CDG) $ \mathcal{G}_{N+1} $ to represent the conflict relationship.

\begin{definition}[Conflict Directed Graph]
	The conflict directed graph is denoted as $ \mathcal{G}_{N+1} = \left\{ \mathcal{V}_{N+1},\mathcal{E}_{N+1} \right\} $. If there are $ N $ vehicles in the control zone, we have vertex set $ \mathcal{V}_{N+1} = \{0,1,2,\dots,N\} $ (vertex $0$ stands for the virtual leading vehicle). Unidirectional edge set is defined as $ \mathcal{E}_{N+1}^{u}=\{(i,j) | i<j, i \in \mathcal{D}_{j}\} $, and bidirectional edge set is defined as $ \mathcal{E}_{N+1}^{b}=\{(i,j) | i<j, i \in \mathcal{C}_{j}\} $. Edge set is the union of these two sets as $ \mathcal{E}_{N+1} = \mathcal{E}_{N+1}^{u} \cup \mathcal{E}_{N+1}^{b} $.
\end{definition}

The CDG of the case scenario is drawn in Fig.~\ref{fig:ConflictDirectedGraph}. The vertices in CDG represent the vehicles in the control zone. The black unidirectional edges represent the diverging conflicts. If there exist a unidirectional edge $ (i,j) $, it means that vehicle $ j $ cannot reach the intersection earlier than vehicle $ i $. The red bidirectional edges stand for the crossing conflict, which means these vehicles' arriving sequences can be exchanged. Note that in DFST method, the crossing conflict $ \mathcal{C}_i $ and diverging conflict $ \mathcal{D}_i $ are treated as one union conflict type. In the following paragraph, we will interpret the separation of these different conflict sets helps to improve the original DFST algorithm performance.

It is straightforward that CDG describes all the conflict relationships of vehicles. Previous research has proved that CDG in DFST method has a depth-first spanning tree, as shown in Lemma~\ref{lemma:CDG}. Since this conclusion is drawn with only unidirectional edges in CDG, it still stands for our OPT-DFST method.

\begin{lemma}[\cite{xu2018distributed}]
	\label{lemma:CDG}
	A conflict directed graph has a depth-first spanning tree with the root node vertex 0.
\end{lemma}

It is evident that for a group of vehicles in the control zone, a passing order solution is related to the depth-first spanning tree addressed in Lemma~\ref{lemma:CDG}. The CDG describes all the conflict relationships of the vehicles and the edges specifically describe the conflict type. If the depth-first spanning tree is built according to CDG, a feasible passing order solution can be built. We propose the OPT-DFST method in Algorithm~\ref{algo:DFST:1} and~\ref{algo:DFST:2}.

\begin{algorithm}[tb]
	\caption{Optimized Depth-first Spanning Tree Method}
	\label{algo:DFST:1}
	\begin{algorithmic}[1]
		\Require{Conflict Directed Graph $ \mathcal{G}_{N+1} = \left\{ \mathcal{V}_{N+1},\mathcal{E}_{N+1} \right\} $} 
		\Ensure{Optimized Depth-first Spanning Tree  $ \mathcal{G}_{N+1}' = \left\{ \mathcal{V}_{N+1},\mathcal{E}_{N+1}' \right\} $}\\
		\textbf{initialize}: Set the depth of vertex $ 0 $'s layer $d_0=0$
		\For{$i = 1,2,...,N$}
		\State $ P = \{m \in \mathcal{V}_{N+1} | (m,i) \in \mathcal{E}_{N+1} \text{ and }m < i\} $ \label{algo:DFST:1:1}
		\State $k$ = FIND-OPT-PARENT($ \mathcal{G}_{N+1}' $,$P$) \label{algo:DFST:1:2}
		\State Set vertex $k$ as the parent vertex of $i$ in the graph $ \mathcal{G}_{N+1}' $, add a vertex $i$ and an edge $(k,i)$ to the graph $ \mathcal{G}_{N+1}' $, and set the depth of the vertex $i$ to $d_{k}+1$
		\EndFor
	\end{algorithmic} 
\end{algorithm}

\begin{algorithm}[tb]
	\caption{FIND-OPT-PARENT}
	\label{algo:DFST:2} 
	\begin{algorithmic}[1] 
		\Require{$ \mathcal{G}_{N+1}' $,$P$} 
		\Ensure{Optimal Parent Vertex $k$}
		\For{$j = 1,2,...,|P|$} 
		\State Find the largest depth $d_\text{div}$ of the diverging conflict parents in $ \mathcal{G}_{N+1}' $
		\State Find the union depths set of crossing conflict parents $D_\text{swap}$ in $ \mathcal{G}_{N+1}' $
		\EndFor
		\State Find $\min{d_k}$, s.t.$\{ k \in P |d_{k}+1 > d_\text{div}, \text{ and } (d_{k}+1) \cap D_\text{swap} = \varnothing\}$ \label{algo:DFST:2:1}
		\State \Return $k$
	\end{algorithmic} 
\end{algorithm}

We design an iterative algorithm to generate a depth-first spanning tree. Similar to the general graph theory, the depth of each node in the spanning tree is calculated by its distance to the root node $ 0 $. The vehicle's conflict-free attribute of the same depth has been proofed in Lemma~\ref{lemma:ConflictFree}. Specifically, the depth of the nodes represent the passing order of the vehicles, \ie, the vehicles in the same depth shall pass the intersection simultaneously. It can be inferred that the depth of the spanning tree is related to the traffic mobility at the intersection. \emph{Evacuation time} means the overall time cost for all the vehicles to pass the intersection, \ie, the leaving time of the last vehicle. It also means the largest depth of the spanning tree nodes, denoted as $ d_\text{all} $ in rest of the paper.

\begin{lemma}[\cite{xu2018distributed}]
	\label{lemma:ConflictFree}
	Consider a virtual platoon characterized by conflict directed graph $ \mathcal{G}_{N+1} $ with the spanning tree $ \mathcal{G}_{N+1}' $. The trajectories of two vehicles with the same depth in $ \mathcal{G}_{N+1}' $ have no conflict relationship with each other.
\end{lemma}

For each node $ i $, we firstly find all the parent vertices $P$ of it in $ \mathcal{G}_{N+1} $. Note that both original DFST method and OPT-DFST method are based on FIFO principle. Even though there are bidirectional edges in CDG, the larger-index vehicles should not be considered in smaller-index vehicles' passing order optimization process. Hence, only the smaller-index vehicles are selected as the parent node candidates, as shown in Line~\ref{algo:DFST:1:1} in Algorithm~\ref{algo:DFST:1}.

Then in Line~\ref{algo:DFST:1:2},~\ie, Algorithm~\ref{algo:DFST:2}, we treat conflict types differently. If the parent node $ j $ is in vehicle $ i $'s diverging conflict set, \ie, $ j \in \mathcal{D}_i $, it means $ j $ is the vehicle in front of $ j $, which also means $ i $ can not surpass $ j $ (Overtaking is not permitted). In this case, we find all the largest depth $ d_\text{div} $ of all the parent vehicles, because vehicle $ i $'s target depth should not surpass any of these vehicles. In other cases, the parent node $ j $ is in vehicle $ i $'s crossing conflict set $ \mathcal{C}_i $. It means vehicle $ i $ and vehicle $ j $ cannot arrive at the intersection simultaneously. Either vehicle $ i $ or vehicle $ j $ can pass the intersection first, so we find the union depth set of crossing conflict parents $ D_\text{swap} $. To find the optimal parent $ k $, we should select the proper depth of vehicle $ i $. Since the depth of parent $ k $ is $ d_k $, the depth of vehicle $ i $ is $ d_{k}+1 $. As mentioned before, $ d_{k}+1 $ should neither smaller than $ d_\text{div} $, nor have intersection with set $ D_\text{swap} $. For traffic efficiency consideration, the depth should be as small as possible, as shown in Line~\ref{algo:DFST:2:1} in Algorithm~\ref{algo:DFST:2}.

In Example~\ref{exp:1}, $ \mathcal{C}_{1} = \varnothing, \mathcal{D}_{1} = \{0\} $, $ d_\text{div} = 0, D_\text{swap}= \varnothing$, so the parent node of vehicle $ 1 $ is selected as node $ 0 $ and $ d_{1} = 1 $. Node $ 2 $ is arranged the same as node $ 1 $. $ \mathcal{C}_{3} = \{1,2\}, \mathcal{D}_{3} = \{0\} $, $ d_\text{div} = 0, D_\text{swap}= \{1\}$, so the parent node of vehicle $ 3 $ is selected as node $ 1 $ and $ d_{3} = 2 $. $ \mathcal{C}_{4} = \{2,3\}, \mathcal{D}_{4} = \{0\} $, $ d_\text{div} = 0, D_\text{swap}= \{1,2\}$, so the parent node of vehicle $ 4 $ is selected as node $ 3 $ and $ d_{4} = 3 $. $ \mathcal{C}_{5} = \{1,4\}, \mathcal{D}_{5} = \{0\} $, $ d_\text{div} = 0, D_\text{swap}= \{1,3\}$, so the parent node of vehicle $ 5 $ is selected as node $ 2 $ and $ d_{5} = 2 $. $ \mathcal{C}_{6} = \{1,4\}, \mathcal{D}_{5} = \{0,5\} $, $ d_\text{div} = 2, D_\text{swap}= \{1,3\}$, so the parent node of vehicle $ 5 $ is selected as node $ 0 $ and $ d_{5} = 4 $. Based on this method we can obtain the optimized depth-first spanning tree $ \mathcal{G}_{N+1}' $ as shown in Fig.~\ref{fig:SpanngTree_Optimized}.

\begin{figure}[tb]
	\centering
	\subcaptionbox{DFST Spanning Tree with $ d_\text{all} = 5 $\label{fig:SpanngTree_Old}}
	{\includegraphics[width=0.75\linewidth]{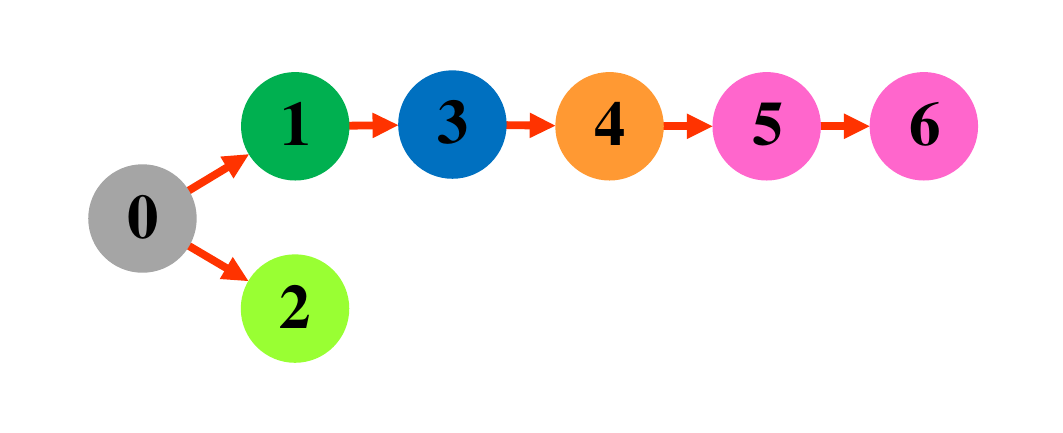}}
	\subcaptionbox{OPT-DFST Spanning Tree with $ d_\text{all} = 4 $ \label{fig:SpanngTree_Optimized}}
	{\includegraphics[width=0.65\linewidth]{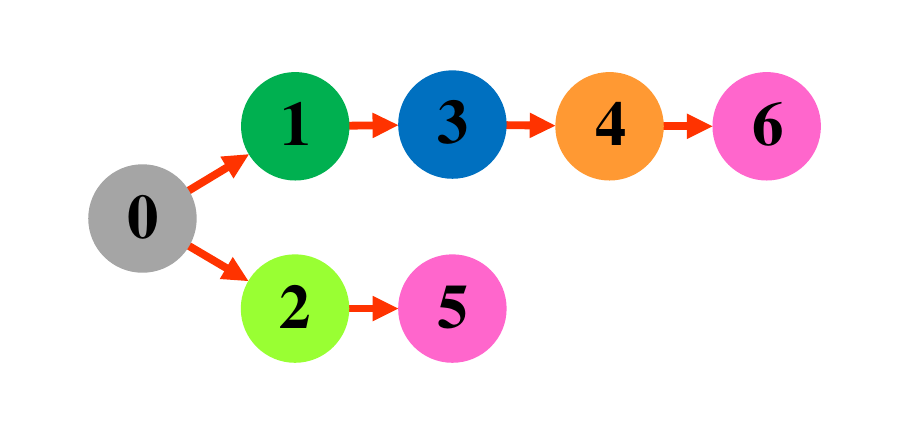}}
	\subcaptionbox{MM Spanning Tree with $ d_\text{all} = 3 $\label{fig:SpanngTree_MM}}
	{\includegraphics[width=0.54\linewidth]{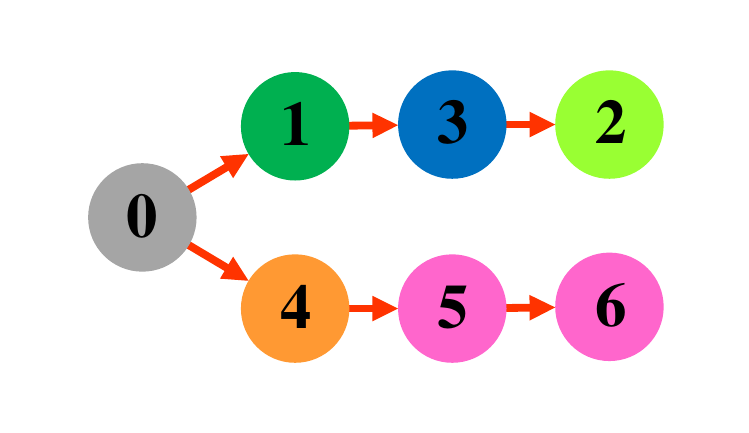}}
	\caption{Different spanning tree results comparison of three different method.}
	\label{fig:SpanningTree}
\end{figure}

The main improvement of OPT-DFST method is on Line~\ref{algo:DFST:1:2}, \ie, Algorithm~\ref{algo:DFST:2}. In \cite{xu2018distributed}, the diverging conflict and crossing conflict are treated as one union conflict set. In this way, the largest depth $ d_{k} $ of the parent vertices are found in Line~\ref{algo:DFST:1:2}. For instance, in our scenario case in Fig.~\ref{fig:SpanngTree_Old}, vehicle $ 5 $ is ranked $ d_{5} = 2 $ in OPT-DFST method. But in DFST method, the parent nodes of vehicle $ 5 $ are nodes $ \{0,1,4\} $, which cause the largest depth $ d_{k} $ of the parent vertices is $ d_{k} = 3 $, so DFST method gets $ d_{5} = 4 $ as shown in Fig.~\ref{fig:SpanngTree_Old}. As a result, the total depth of $ 6 $ vehicles to $ d_\text{all} = 5 $ while the OPT-DFST method optimized it into $ d_\text{all} = 4 $.

\begin{remark}
We come to a conclusion that the DFST method basically focus on the feasibility of the passing order problem. The proposed OPT-DFST method further considers the optimality,~\ie, tries to find the lowest depth for each vehicle, but still, the solution is found vehicle by vehicle,~\ie, based on FIFO principle. In other words, the solution is a local optimal solution in arranging each vehicle. Since each vehicle need to check all its conflict relationship of the parent nodes, the computational complexity is $ O(N^2) $.
\end{remark}

\subsection{Maximum Matching Method}
\label{Sec:MM}
As mentioned before, we aim to minimize the overall depth $ d_\text{all} $ of the spanning tree, which corresponds to the evacuation time of all the vehicles. If we re-consider the case scenario in Fig.~\ref{fig:ScenarioSimplify}, the right-turn traffic flows have no conflict with any other traffic flows. Considering the rest eight kinds of left-turn or straight traffic flows, a maximum number of two vehicles can pass the intersection simultaneously. From this point of view, another method to describe the conflict relationships of the vehicle is to describe coexisting relationships of them.

\begin{definition}[Coexisting Undirected Graph]
	The coexisting undirected graph is defined as the complement graph of the conflict directed graph $ \mathcal{G}_{N+1} $ excluding node $ 0 $. It denotes $ \overline{\mathcal{G}_{N}} = \{\overline{\mathcal{V}_{N}},\overline{\mathcal{E}_{N}}\} $, where $ \overline{\mathcal{V}_{N}} = \mathcal{V}_{N+1} - \{0\} $, $ \overline{\mathcal{E}_{N}} = \{(i,j) | i,j \in \overline{\mathcal{V}_{N}} \text{ and } (i,j) \notin \mathcal{E}_{N+1}\} $.
\end{definition}

In the case scenario, the conflict directed graph (CDG) is drawn in Fig.~\ref{fig:ConflictDirectedGraph} and coexisting undirected graph (CUG) in Fig.~\ref{fig:CoexistUndirectedGraph}. Since CDG edges $ \mathcal{E}_{N+1} $ means two vehicles have conflicts and CUG is the complement graph of CDG, CUG edges $ \overline{\mathcal{E}_{N}} $ means two vehicles are conflict-free, \ie, can pass the intersection simultaneously. 

Recall that we expect to minimize the overall depth $ d_\text{all} $ of the spanning tree. Since the total vehicle number, \ie, the node number $ \left| \overline{\mathcal{V}_{N}} \right|$ is a constant value, minimizing the overall depth of the spanning tree equals to widening the average width of the spanning tree. Thus, solving the optimal passing solution is equivalent to find the maximum pairings of coexisting vehicles combination in CUG. Note that this conclusion corresponds to common sense. Maximum pairings of vehicles mean maximizing the usage of the intersection, which lowers the overall evacuation time.

In graph theory, maximum matching (MM) is a well-developed theory that is capable of solving this problem~\cite{edmonds1965maximum}. Firstly, we define matching in our case formally.

\begin{definition}[Matching]
	A matching $ \mathcal{M} $ in $ \overline{\mathcal{G}_{N}} $ is a subset of its edge $ \overline{\mathcal{E}_{N}} $, such that no vertex in $ \overline{\mathcal{V}_{N}} $ is incident to more that one edge in $ \mathcal{M} $.
\end{definition}

The matching $ \mathcal{M} $ is the aforementioned pairing set. Since the edges are selected from CUG $ \overline{\mathcal{G}_{N}} $, the two vehicles are conflict-free in passing the intersection simultaneously. Because no two edges in $ \mathcal{M} $ have the same vertex, each vehicle only appears once in the matching. In Example~\ref{exp:1}, $ \mathcal{M} = \{(2,6),(3,5)\} $ is a matching, but it is still not a maximal matching because $ (1,4) $ can be add into the matching. $ \mathcal{M} = \{(1,4),(2,6),(3,5)\} $ is a maximal matching, but still we cannot claim that it is the maximum number of pairing in CUG. Our target is to find the maximum number of pairing in CUG, \ie, the maximum number of pairing, so we further introduce maximum matching.

\begin{definition}[Maximum Matching]
	\label{Def:MM}
	A matching $ \mathcal{M} $ is said to be maximum if for any other matching $ \mathcal{M}' $, $ \left|\mathcal{M}\right| > \left|\mathcal{M}'\right|$. The maximum matching is denoted as $ \mathcal{M}_{\max} $ in rest of the paper.
\end{definition}

Maximum matching $ \mathcal{M}_{\max} $ is the maximum sized matching in CUG $ \overline{\mathcal{G}_{N}} $, which corresponds to find the maximum number of vehicle pairing. In the case scenario, the maximum matching $ \mathcal{M}_{\max} = \{(1,4),(2,6),(3,5)\} $ with size of $ 3 $. The maximum matching have been solved in time $ O(\left|\overline{\mathcal{E}_{N}}\right| \left|\overline{\mathcal{V}_{N}}\right|^2) $ using Edmonds' blossom algorithm~\cite{edmonds1965maximum}.

After maximum matching, we can add the rest single vehicles in the end, forming the optimized feasible spanning tree which is the same as shown in Fig.~\ref{fig:SpanngTree_Optimized}. The process which completes the spanning tree is introduced as follows, denoting as SPANNING process in Line~\ref{algo:MM:2}. CUG $ \overline{\mathcal{G}_{N}} $ only contains the coexisting information of the vehicles, while the vehicle's passing order from the same lane should be strictly according to their relative position. The MM method could generate solutions that cannot be directly executed, \eg, $ \mathcal{M}_{\max} = \{(1,4),(3,6),(2,5)\} $ with size of $ 3 $. The spanning tree forms as $ (1,4) \rightarrow (3,6) \rightarrow (2,5)$. However, this solution is not feasible since vehicle $ 5 $ is in front of vehicle $ 6 $. We proof that this can be solved by exchange the unfeasible sequence as shown in Theorem~\ref{The:1}.

\begin{theorem}
	\label{The:1}
	If the maximum matching $ \mathcal{M}_{\max} $ contains $ \{(i,j), (n,m)\} $ that is not feasible, $ \{(i,m), (j,n) \}$ must be a feasible solution.
\end{theorem}
\begin{proof}
	Without loss of generality, we assume in the maximum matching solution, $ m $ and $ j $ are in the same lane with $ m $ in front of $ j $, causing the unfeasibility. According to the definition of CDG, there exist a unidirectional edge $ (m,j) $ in CDG. If $ (i,m), (j,n) $ is still not a feasible solution, the unfeasibility can only be caused by two possibilities. The first case is that $ (i,j) $ or $ (n,m) $ is in the same lane which is contradicted to the maximum matching solution. Second case is that $ (m,j) $ are in the same lane with $ j $ in front of $ m $, which means there exist a unidirectional edge $ (j,m) $ in CDG and it is not possible. As a result,  $ (i,m), (j,n) $ must be a feasible solution.
\end{proof}

If the original MM solution is $ \mathcal{M}_{\max} = \{(1,4),(3,6),(2,5)\} $, we can exchange vehicle $ 5 $ and $ 6 $, obtaining the feasible solution $ \mathcal{M}_{\max} = \{(1,4),(3,5),(2,6)\} $. 

\begin{remark}
	We sum up the MM method in Algorithm~\ref{algo:MM}. Recall that in the process of solving the passing order, the vehicles' arrival sequence is neglected on purpose. Therefore, the FIFO principle is abandoned and the fundamental vehicle conflict relationship is highlighted, which refers that MM method obtains the global optimal solution as indicating by Definition~\ref{Def:MM}. Since the maximum matching in Line~\ref{algo:MM:1} takes $ O(\left|\overline{\mathcal{E}_{N}}\right| \left|\overline{\mathcal{V}_{N}}\right|^2) $ time and the adjusting and spanning in Line~\ref{algo:MM:2} takes $ O(\overline{\mathcal{V}_{N}}) $, the overall MM method takes $ O(\left|\overline{\mathcal{E}_{N}}\right| \left|\overline{\mathcal{V}_{N}}\right|^2 + \overline{\mathcal{V}_{N}}) = O(N^4) $ time.
\end{remark}

\begin{algorithm}[tb]
	\caption{Maximum Matching Method}
	\label{algo:MM} 
	\begin{algorithmic}[1] 
		\Require{Coexisting Undirected Graph $\overline{\mathcal{G}_{N}}$} 
		\Ensure{Spanning Tree $ \mathcal{G}_{N+1}' $}
		\State $ \mathcal{M}_{\max} $ = MAXIMUM-MATCHING($ \mathcal{G}_{N+1}' $) \label{algo:MM:1} 
		\State $ \mathcal{G}_{N+1}' $ = SPANNING($ \mathcal{M} $) \label{algo:MM:2}
		\State \Return $ \mathcal{G}_{N+1}' $
	\end{algorithmic} 
\end{algorithm}

\subsection{Virtual Platooning Topology}
As mentioned before, all the three method,~\ie, DFST, OPT-DFST and MM, lead to a feasible spanning tree $ \mathcal{G}_{N+1}' = \left\{ \mathcal{V}_{N+1},\mathcal{E}_{N+1}' \right\} $ as shown in Fig.~\ref{fig:SpanngTree_Optimized}, which forms the geometry topology of the virtual platoon. Since each node $ \mathcal{V}_{i} $ in $ \mathcal{G}_{N+1}' $ represents a vehicle $ i $ and the edge $ \mathcal{E}_{i}' = (j,i) $ stands for the virtual preceding vehicle $ j $, the vehicle can always follow its parent node vehicle $ \mathcal{V}_{j} $. In other words, vehicle $ i $ and its parent vehicle $ j $ intend to keep a constant desired car-following distance $ D_{des} $ and the same velocity $ v_{des} $.

\begin{equation}
	\left\{
		\begin{array}{l}
			\displaystyle\lim_{t \rightarrow \infty}\left\|v_{i}(t)-v_{j}(t)\right\|=0 \\
			\displaystyle\lim_{t \rightarrow \infty}\left(p_{j}(t)-p_{i}(t)-D_{des}\right)=0
		\end{array}, i \in\{1,2, \cdots, N\},  \right.
\end{equation}
in which $ v_{i}(t) $ and $ p_{j}(t) $ denote the velocity and position of vehicle $ i $. Lemma~\ref{lemma:ConflictFree} has proven that the vehicles in the same depth have no conflict relationship. Therefore, if vehicle $ i $ and vehicle $ j $ are of the same depth, it stands

\begin{equation}
	\displaystyle\lim_{t \rightarrow \infty}\left(p_{i}(t)-p_{j}(t)\right)=0.
\end{equation}

The state of vehicle $ i $ is defined as $ x_{i} = \left[p_{i};v_{i}\right] $ and second-order model is employed to describe the vehicle dynamics. As for the low-level vehicle control, a linear feedback PID controller is designed to execute the virtual platooning~\cite{peppard1974string}. 

\section{SIMULATIONS}
\label{Sec:Simulation}
\subsection{Simulation Environment}
The traffic simulation is conducted in SUMO, where we use the default electric vehicle parameters. The intersection scenario and the lane direction settings are the same as shown in Fig.~\ref{fig:Intersection}, and the control zone length $ R_c = 1000 m $. The vehicle arrival is assumed to be binomially distributed flow with probability $ p $.

\subsection{Simulation under Different Vehicle Number}
\begin{figure}[tb]
	\centering
	\subcaptionbox{Maximum depth of the spanning tree of three methods.\label{fig:VehiceNum:1}}
	{\includegraphics[width=0.87\linewidth]{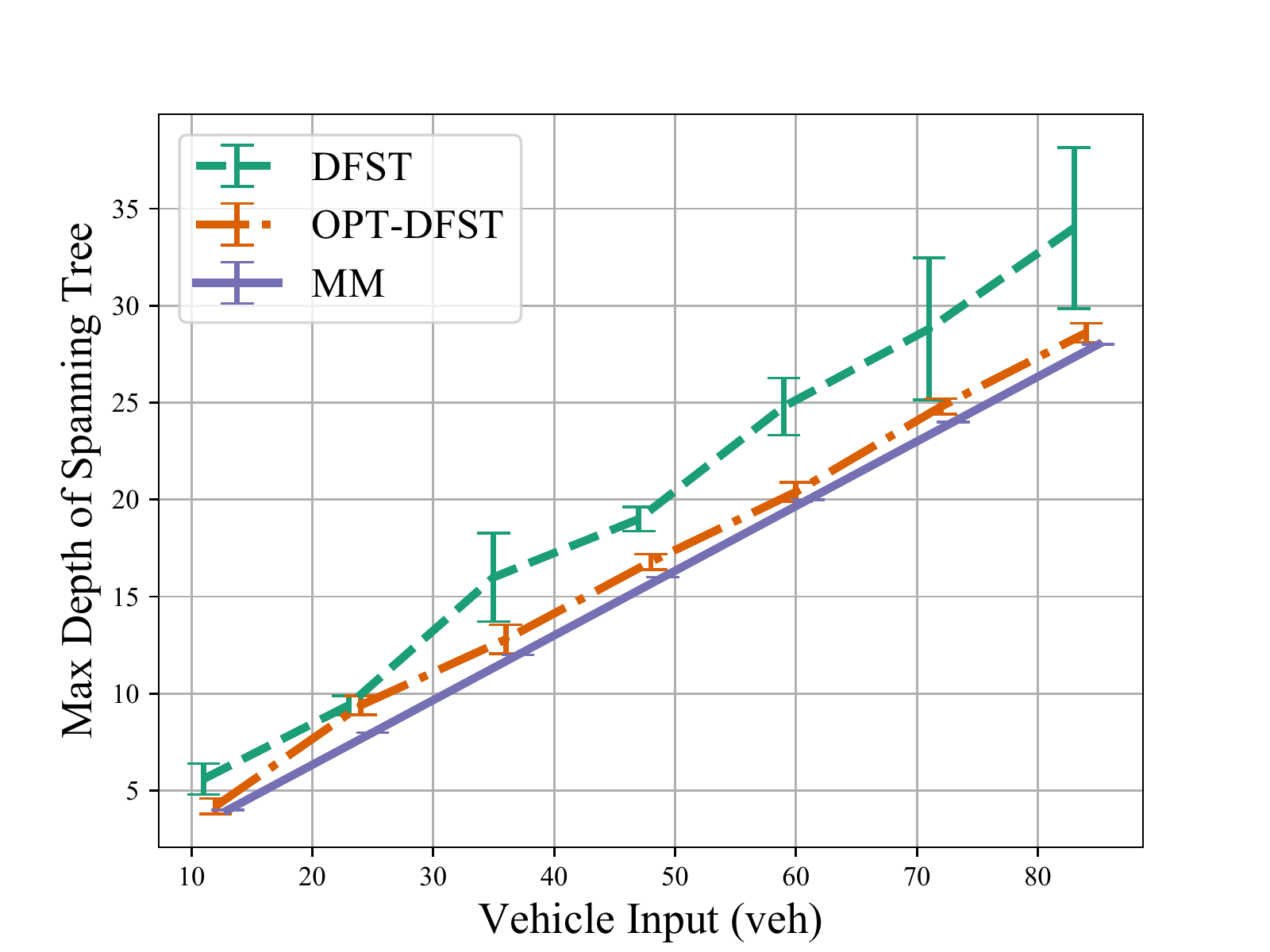}}
	\subcaptionbox{Evacuation time comparison.\label{fig:VehiceNum:2}}
	{\includegraphics[width=0.87\linewidth]{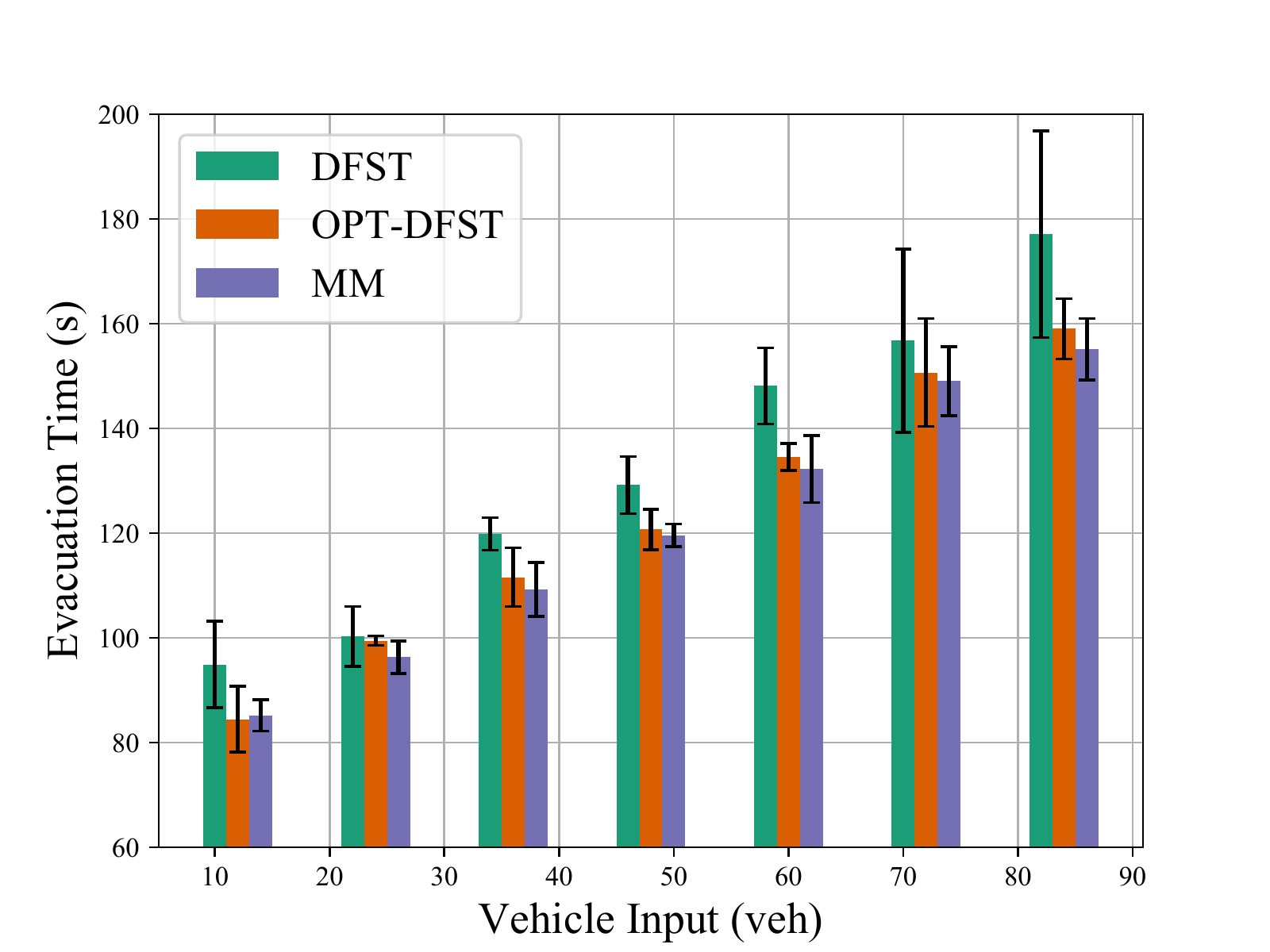}}
	\caption{Simulation results of different vehicle input number.}
	\label{fig:VehiceNum}
\end{figure}

Since obtaining the optimal passing order solution is the critical point of this research, the first simulation is conducted with different vehicle numbers $ N $. The vehicle arrival probability is $ p = 0.3 $, and the maximum vehicles input is $ N = 84 $. Generally speaking, there are maximum $ 12^{72}\cdot12! $ possible passing order solutions in $ 12 $ lanes scenario.

Fig.~\ref{fig:VehiceNum} demonstrates the performance of different algorithms. Each vehicle input consists of five-time simulations, where input flow is randomly distributed according to $ p $. Standard deviations of the result are also presented as error bar in Fig.~\ref{fig:VehiceNum}. Firstly, the OPT-DFST method outperforms the original DFST method. By separating the conflict types as explained in Section~\ref{Sec:OPT-DFST}, OPT-DFST is capable of finding the local optimal passing order solution,~\ie, smaller depth of the spanning tree for each vehicle. Although both of these two methods are based on the FIFO principle, the OPT-DFST method obtains a better passing order solution considering the total evacuation time. However, because of the FIFO principle, OPT-DFST only searches for the local optimal solution as each new vehicle comes. In contrast, the MM method is capable of searching the global optimal solution of all passing order possibilities by discarding the arrival sequence of the existing vehicles. Besides, the error bars show that the MM method has the best performance in standard deviations,~\ie, finding the global optimal solution especially when the vehicle input number is rather large. As a result, in the situation of $ 84 $ input vehicles, OPT-DFST saves $ 10.2\% $ evacuation time than that of DFST, and MM further saves $ 2.2\% $ of it.

\section{CONCLUSIONS}
\label{Sec:Conclusion}
This paper presents a graph-based vehicle cooperation method to guarantee safety and efficiency at unsignalized intersections. Base on graphic analysis, conflict directed graphs and coexisting directed graphs are drawn to model the conflict relationship of the vehicles. Based on the existing research, an optimized depth-first spanning tree method is proposed, which proves to be the local optimal passing order solution. The maximum matching method is further developed to find the global optimal solution. The computational complexity of both methods is provided and traffic simulation verified the effectiveness of the proposed methods.




%

%


\bibliographystyle{IEEEtran}
\bibliography{IEEEabrv,mybibfile}
	
\end{document}